\documentclass{elsarticle}
\usepackage{lineno,hyperref}
\usepackage{amsmath,amssymb,amsthm}
\usepackage{bm}
\usepackage{color}
\usepackage{amscd}
\usepackage[all]{xy}
\usepackage{tikz}
\usepackage{bm}
\modulolinenumbers[5]

\journal{Journal of \LaTeX\ Templates}

\newtheorem{thm}{Theorem}[section]
\newtheorem{defi}[thm]{Definition}

\newtheorem{prop}[thm]{Proposition}

\newtheorem{lem}[thm]{Lemma}
\newtheorem{exam}[thm]{Example}

\newcommand{\GL}{{\rm GL}}

\newcommand{\Hom}{{\rm Hom}}

\newcommand{\Q}{{\mathbb{Q}}}

\newcommand{\X}{{\mathbb{X}}}
\newcommand{\C}{{\mathbb{C}}}
\newcommand{\calB}{{\mathcal{B} }}

\newcommand{\calT}{{\mathcal{T} }}
\newcommand{\calL}{{\mathcal{L} }}
\newcommand{\calR}{{\mathcal{R} }}
\newcommand{\calP}{{\mathcal{P} }}
\newcommand{\calQ}{{\mathcal{Q} }}

\newcommand{\ord}{\mbox{\rm ord}}

\newcommand{\bfc}{\mathbf{c}}

\newcommand{\linrel}{\mbox{\rm LinRel}}
\newcommand{\rem}{\mbox{\rm rem}}

\begin{document}

\begin{frontmatter}
\title{Hilbert's Irreducibility Theorem for Linear Differential Operators}
\tnotetext[mytitlenote]{This work was supported by the National Key Research Project of China under grant No. 2023YFA1009401 and No. 2020YFA0712300, and by Postdoctoral Fellowship Program of CPSF (GZC20230750).}
\author{Ruyong Feng, Zewang Guo}

\address{KLMM,Academy of Mathematics and Systems Science, Chinese Academy of Sciences and School of Mathematics, University of Chinese Academy of Sciences, 100190, Beijing\\
ryfeng@amss.ac.cn, guozewang@amss.ac.cn}

\author{Wei Lu}

\address{
Hubei Key Laboratory of Applied Mathematics, Faculty of Mathematics
and Statistics, Hubei University, Wuhan 430062, China.\\
weilu23@hubu.edu.cn.}


\begin{abstract}
We prove a differential analogue of Hilbert's irreducibility theorem. Let $\calL$ be a linear differential operator with coefficients in $C(\X)(x)$ that is irreducible over $\overline{C(\X)}(x)$, where $\X$ is an irreducible affine algebraic variety over an algebraically closed field $C$ of characteristic zero. We show that the set of $c\in \X(C)$ such that the specialized operator $\calL^c$ of $\calL$ remains irreducible over $C(x)$ is Zariski dense in $\X(C)$.
\end{abstract}


\end{frontmatter}
\biboptions{sort&compress}
\section{Introduction}
Linear differential operators, as typical non-commutative polynomials, share many algebraic properties with usual polynomials. Concepts such as the Euclidean algorithm, irreducibility, Eisenstein's criterion,  greatest common divisor, and least common multiple from the polynomial ring have counterparts in the realm of linear differential operators. Readers can refer to \cite{Ore1933} for an algebraic framework for skew polynomials, including not only linear differential operators but also linear difference and q-difference operators. Similar to the polynomial case, testing irreducibility and factorization are fundamental tasks in the algorithmic aspect of linear differential operators. However, factorizing linear differential operators is more complicated, as they may have infinitely many irreducible divisors, and their factorization is not usually unique. To address the uniqueness of factorization, one has to introduce the notion of similarity (see \cite{Ore1933} for the definition and \cite{Jacobson1996, LiWang2011} for methods of testing similarity). Despite these complexities, various methods have been developed for the aforementioned tasks, see, for example, \cite{Grigoriev1990,Tsarev1994,vanHoeij1997,CompointSinger:computinggaloisgroups1999,LiSchwarzTsarev2003,CompointWeil2004} for algorithms of factorization and see \cite{Kovacic1972,Grigoriev1990,Singer1996,ChurchillZhang2009} for methods of testing irreducibility. 

In this paper, our emphasis is not on the development of algorithms for testing irreducibility or factorization. Instead, we concentrate on addressing the following problem: given an irreducible linear differential operator with parameters, describe the set of values for which the specializations of the linear differential operator at these values remain irreducible. This can be viewed as a differential analog of Hilbert's irreducibility theorem. Let $f \in \mathbb{Q}[\mathbf{x}, y]$ be an irreducible polynomial, where $\mathbf{x} = {x_1, \dots, x_m}$.
Hilbert's irreducibility theorem asserts that there exist infinitely many $\mathbf{c} \in \mathbb{Q}^m$ such that $f(\mathbf{c}, y)$ remains irreducible.  Fields where Hilbert's irreducibility theorem holds are referred to as Hilbertian fields. For precise definition, readers can refer to Chapter 12 of \cite{FriedJarden2008}.
By leveraging Hilbert's irreducibility theorem, one can effectively reduce the inverse problem of classical Galois theory over $\mathbb{Q}$ to the problem over $\mathbb{Q}(t)$. Notably, Hilbert demonstrated that both symmetric and alternating groups can be realized as the Galois groups of polynomials over $\mathbb{Q}$.  

In the context of this paper, we present an irreducibility theorem for linear differential operators. Precisely, assume that $\X$ is an irreducible affine algebraic variety over an algebraically closed field $C$ of characteristic zero, and $C(\X)$ stands for the field of all rational functions on $\X$. Let $C(\X)(x)$ be the differential field with usual derivative $\delta=\frac{d}{dx}$, and let $C(\X)(x)[\delta]$ be the ring of differential operators with coefficients in $C(\X)(x)$, where $\delta x=x\delta +1$. As usual, $\X(C)$ denotes the set of $C$-points of $\X$, which is identified with $\Hom_C(C[\X],C)$. For $c\in \X(C)$ and $f\in C[\X]$, the ring of regular functions on $\X$,  we shall denote by $f^c$ the image of $f$ under the homomorphism $c$.  Furthermore, if $f\in C[\X][x]$ then $f^c$ represents the polynomial obtained by applying $c$ to the coefficients of $f$. Set 
\begin{equation}
\label{eqn:differenceoperator}
    \calL=a_n(x) \delta^n+a_{n-1} (x)\delta^{n-1}+\dots+a_0(x)
\end{equation}
where $a_i(x)\in C(\X)(x)$ and $a_n(x)\neq 0$. For $c\in \X(C)$, $\calL^c$ denotes the operator in $C(x)[\delta]$ obtained by applying $c$ to the coefficients of $\calL$. It is said that $\calL^c$ is well-defined if the denominators of all coefficients of $\calL$ do not vanish under the application of $c$ and $a_n^c\neq 0$. The main result of this paper is the following theorem.
\begin{thm}
\label{thm:mainresult}
 Assume that $\calL\in C(\X)(x)[\delta]$ is irreducible over $\overline{C(\X)}(x)$. Then there exists an ad-open subset $U$ of $\X(C)$ such that for any $c\in U$, $\calL^c$ is well-defined and $\calL^c$ is irreducible over $C(x)$.
\end{thm}
Let's give an example to illustrate the above theorem.
\begin{exam}
Consider the linear differential operator
$$
    \calL=\delta^2+\frac{1}{x}\delta+\frac{x^2-\alpha^2}{x^2}
$$
corresponding to Bessel's differential equation, where $\alpha$ is a parameter. The above operator is irreducible over $\overline{\C(\alpha)}(x)$. From Example 5.8 of \cite{Feng-Wibmer:differentialgaloisgropus}, for $c\in \C$, $\calL^c=\delta^2+\frac{1}{x}\delta+\frac{x^2-c^2}{x^2}$ is irreducible over $\C(x)$ if and only if $c-\frac{1}{2}$ is not an integer. In particular, if $c\notin \Q$ then $\calL^c$ is irreducible over $\C(x)$, and obviously the set $\C\setminus \Q$ is Zariski dense in $\C$. 
\end{exam}

The concept of ad-open sets was initially introduced by Hrushovski in \cite{Hrushovski:computingdifferentialgaloisgroups} and further explored in \cite{Feng-Wibmer:differentialgaloisgropus}. Alongside ab-open sets, it serves to characterize the set of $c \in \X(C)$ for which the specialization of the Galois group of $\calL$ under $c$ precisely matches the Galois group of $\calL^c$. Additionally, it has been demonstrated in \cite{Hrushovski:computingdifferentialgaloisgroups, Feng-Wibmer:differentialgaloisgropus} that every ad-open subset of $\X(C)$ is Zariski dense in $\X(C)$. The proof of Theorem~\ref{thm:mainresult} relies on a result from \cite{Feng-Wibmer:differentialgaloisgropus}, asserting the existence of a positive integer $N$ and an ad-open subset $U$ of $\X(C)$ such that the certificates for all exponential solutions of $\calL^c(y)=0$ have a degree not exceeding $N$ for all $c \in U$.

The paper is organized as follows. In Section 2, we shall show that there exists a uniformly bound for the degrees of irreducible right-hand divisors of $\calL$ and this degree bound is well-behaved under specializations.  In Section 3, we shall prove Theorem~\ref{thm:mainresult}.

\section{Degree bound for the coefficients of irreducible divisors}
\label{sec:ad-opensets}
In this section, leveraging results from \cite{Feng-Wibmer:differentialgaloisgropus}, we are going to show that  there exists a positive integer $N$ and an ad-open subset $U$ of $\X(C)$ such that for all $c\in U$, the coefficients of every irreducible right-hand divisor of $\calL^c$ are uniformly bounded by $N$. Let's begin by recalling the notion of ad-open sets.
\begin{defi}
Let $\Gamma$ be a finitely generated subgroup of $(C[\X],+)$. The set 
$$
   \calB(\Gamma, \X)=\{ c\in \X(C) \mid \mbox{$c$ is injective on $\Gamma$}\}
$$
is called a basic ad-open subset of $\X$. An ad-open subset of $\X(C)$ is defined to be an intersection of finitely many basic ad-open subsets of $\X(C)$. The complement of an ad-open subset is called an ad-closed subset.
\end{defi}
Assume that $U$ is a principal Zariski open subset of $\X(C)$ defined by a nonzero $p\in C[\X] $. Let $\Gamma$ be the subgroup of $(C[\X],+)$ generated by $p$. Then $U=\calB(\Gamma,\X)$. Therefore, every non-empty Zariski open subset of $\X(C)$ is ad-open. As shown in Example 2.22 of \cite{Feng-Wibmer:differentialgaloisgropus}, every finite dimensional $\Q$-vector subspace of $C[\X]$ is ad-closed, for instance, $\Q$ is an ad-closed subset of $\C$.

Let $k$ be an algebraically closed field of characteristic zero, and consider $\calL\in k(x)[\delta]$.  Write
\begin{equation}
\label{eqn:lineardifferenceoperator}
   \calL=b_n\delta^n+b_{n-1}\delta^{n-1}+\dots+b_0,\,\,b_i\in k(x), b_n\neq 0.
\end{equation}
For convenience, we introduce the following definition. 
\begin{defi}
We call $\mathop{\max}\limits_{0\leq i \leq n-1} \left\{\deg\left(\frac{b_i}{b_n}\right)\right\}$ the degree of $\calL$, denoted by $d(\calL)$.
\end{defi}
Suppose that $\calL_1\in k(x)[\delta]$ is an irreducible right-hand divisor of $\calL$, meaning that $\calL=\calL_2\calL_1$ for some $\calL_2\in k(x)[\delta]$ and $\calL_1$ can not be factored further.  Our goal is to establish a bound for $d(\calL_1)$ and investigate its behavior under specializations. It is worth noting that a degree bound in terms of the order and bit-size of $\calL$ has been provided in \cite{Grigorev1989,Grigoriev1990}. To investigate the behavior of this bound under specializations, we now present an alternative degree bound for irreducible right-hand divisors.
\begin{defi}
A positive integer $N$ is called an exponential bound for $\calL$ if $\deg(a)\leq N$ for any $a\in k(x)$ such that $\delta-a$ is a right-hand divisor of $\calL$. We shall use $N(\calL)$ to denote an exponential bound for $\calL$.
\end{defi}
Note that $\delta-a$ is a right-hand divisor of $\calL$ if and only if $\calL$ has an exponential solution $h$ with $a=\delta(h)/h$. Therefore, if $N$ is an exponential bound of $\calL$ then for any exponential solution $h$ of $\calL$, the certificate $\delta(h)/h$ is of degree not greater than $N$.  We shall establish a bound for the degrees of all irreducible right-hand divisors of $\calL$ in terms of an exponential bound for some exterior system of $\calL$.  Let $\calL$ be as in (\ref{eqn:lineardifferenceoperator}), and let $A_\calL$ denote the companion matrix of $\calL$, i.e. 
\[
  A_\calL= \begin{pmatrix} 
        0 & 1 & &  &  \\ & 0 & 1 & & \\  & & \ddots & \ddots & \\ & & & 0 & 1 \\ -\frac{b_0}{b_n} & -\frac{b_1}{b_n} &\dots & \dots & -\frac{b_{n-1}}{b_n}
   \end{pmatrix}.
\]
Let $T_i$ be a transformation matrix that transforms the matrix equation $\delta(Y)=(\wedge^i A_\calL)Y$ into the scale equation $\calL_{[i]}(y)=0$, where $\delta(Y)=(\wedge^i A_\calL)Y$ denotes the $i$th exterior system of $\delta(Y)=A_\calL Y$, see \cite{Grigoriev1990} for the construction. Set 
\begin{equation}
\label{eqn:degreebound}
b(\calL)=\max_{1\leq i\leq n} \left\{2{n \choose i}\deg(T_i)+{n \choose i}\left({n\choose i}-1\right)N(\calL_{[i]})\right\}.
\end{equation}
\begin{prop}
\label{prop:degreebound} 
Suppose that $\calL_1$ is an irreducible right-hand divisor of $\calL$ and $\ord(\calL_1)<\ord(\calL)$. Then 
$
    d(\calL_1)\leq n^3 b(\calL).
$
\end{prop}
\begin{proof}
Let $R$ be the Picard-Vessiot ring over $k(x)$ for $\calL(y)=0$, and $u\in R$ be a nonzero solution of $\calL_1(y)=0$. 
As in \cite{Feng-Wibmer:differentialgaloisgropus}, set 
$$
  \linrel(u,k(x))=\left\{ f\in k(x)[z_1,\dots,z_n] \left | \begin{array}{c}
                                                 \mbox{$f$ is linear homogeneous, and} \\
                                                 \mbox{$f(u,\delta(u),\dots,\delta^{n-1}(u))=0$}\end{array}\right. \right\}.
$$
Due to Lemma 5.11 of \cite{Feng-Wibmer:differentialgaloisgropus}, the vector space $\linrel(u,k(x))$ has a $k(x)$-basis consisting of elements of the form $\alpha_1 z_1+\dots+\alpha_n z_n$ with $\alpha_i\in k(x)$ and 
\begin{equation}
\label{eqn:degreeboundforlinearrelation}
  \deg(\alpha_i)\leq 2{n \choose s}\deg(T_s)+{n \choose s}\left({n \choose s}-1\right)N(\calL_{[s]}) \leq b(\calL),
\end{equation}
where $s=\dim_{k(x)}(\linrel(u,k(x)))$ and $\calL_{[s]}$ is the scale linear differential operator corresponding to the matrix equation $\delta(Y)=(\wedge^s A_\calL)Y$ and $T_s$ is the corresponding transformation as described above. Note that $\beta \calL_1$ is also an irreducible right-hand divisor of $\calL$ for any nonzero $\beta\in k(x)$ and $d(\beta \calL_1)=d(\calL_1)$. We may assume that $\calL_1$ is monic. Write $\calL_1=\delta^\nu+\beta_{\nu-1}\delta^{\nu-1}+\dots+\beta_0$, where $\nu=\ord(\calL_1)$ and $\beta_i\in k(x)$.
Since $\calL_1$ is irreducible, $u, \delta(u), \dots, \delta^{\nu-1}(u)$ are linearly independent over $k(x)$ and $\delta^{\nu}(u)=-(\beta_0u+\dots+\beta_{\nu-1}\delta^{\nu-1}(u))$. Hence $s=n-\nu$.

Let $\{p_i:=\alpha_{i,1}z_1+\dots+\alpha_{i,n}z_n \mid 1\leq i \leq s\}$ be a $k(x)$-basis of $\linrel(u,k(x))$ with the degree bound for $\alpha_{i,j}$ as stated in (\ref{eqn:degreeboundforlinearrelation}). 
One has that $q:=\beta_0z_1+\dots+\beta_{\nu-1}z_\nu+z_{\nu+1}\in \linrel(u,k(x))$. Therefore there exist uniquely $c_1,\dots,c_s\in k(x)$ such that $q=\sum_{i=1}^s c_i p_i$. Consider the system of linear equations
\begin{equation}
\label{eqn:linearsystem}
     (x_1,x_2,\dots,x_s)\underbrace{\begin{pmatrix} \alpha_{1,\nu+1} & \alpha_{1,\nu+2} & \dots & \alpha_{1,n} \\  \alpha_{2,\nu+1} & \alpha_{2,\nu+2} & \dots & \alpha_{2,n} \\ \vdots& \vdots& & \vdots\\ \alpha_{s,\nu+1} & \alpha_{s,\nu+2} & \dots & \alpha_{s,n}\end{pmatrix}}_{B}=(1,0,\dots,0).
\end{equation}
From $q=\sum_{i=1}^s c_i p_i$, it is evident that $(c_1,\dots,c_s)$ is a solution of (\ref{eqn:linearsystem}). Suppose that $(\tilde{c}_1,\dots,\tilde{c}_s)$ with $\tilde{c}_i\in k(x)$ is another solution of (\ref{eqn:linearsystem}) then $0\neq \sum_{i=1}^s (c_i-\tilde{c}_i)p_i \in \linrel(u,k(x))$ is of the form $\gamma_1 z_1+\dots +\gamma_{\nu-1}z_{\nu-1}$. This implies that $u,\delta(u),\dots,\delta^{\nu-1}(u)$ are linearly dependent over $k(x)$, leading to a contradiction. Consequently, $(c_1,\dots,c_s)$ is the unique solution of (\ref{eqn:linearsystem}) and thus the coefficient matrix $B$ is invertible. 

Let's estimate a bound for the $\deg(c_i)$. Set $D$ to be the least common multiply of all denominators of $\alpha_{i,j}$, and write $\alpha_{i,j}=\frac{\tilde{\alpha}_{i,j}}{D}$ with $\tilde{\alpha}_{i,j}\in k[x]$. Then neither $\deg(D)$ nor $\deg(\tilde{\alpha}_{i,j})$ is greater than $s^2\max_{i,j} \{\deg(\alpha_{i,j})\}$. Write $\det(B)=M/D^s$, where $M\in k[x]$ and $\deg(M)\leq s^3 \max_{i,j} \{\deg(\alpha_{i,j})\}$. By Cramer's rule, one has that 
$c_i=\frac{\det(E_i)}{\det(B)}$
, where $E_i$ is the matrix obtained by replacing the $i$th row of $B$ with $(1,0,\dots,0)$. One sees that $\det(E_i)=\frac{\tilde{c}_i}{D^{s-1}}$, where $\tilde{c}_i\in k[x]$ and
$$
  \deg(\tilde{c}_i)\leq (s-1)\max_{i,j}\{\deg(\tilde{\alpha}_{i,j})\}\leq (s-1)s^2\max_{i,j} \{\deg(\alpha_{i,j})\}.
$$ 
Since each $$
\beta_i=\sum_{j=1}^s c_j\alpha_{j,i+1}=\sum_{j=1}^s \frac{D\tilde{c}_j}{M}\frac{\tilde{\alpha}_{j,i+1}}{D}=\frac{\sum_{j=1}^s\tilde{c}_j\tilde{\alpha}_{j,i+1}}{M}
$$ 
and $s<n$,
it follows that
$$
  \deg(\beta_i)\leq n^3\max_{i,j} \{\deg(\alpha_{i,j})\}\leq n^3 b(\calL).
$$ 
Consequently, $d(\calL_1)\leq n^3 b(\calL)$.
\end{proof}

By employing a reasoning analogous to that presented in the proof of Proposition 5.12 of \cite{Feng-Wibmer:differentialgaloisgropus}, we demonstrate that the degree bound stated in Proposition~\ref{prop:degreebound} behaves consistently under specializations.
\begin{prop}
\label{prop:boundunderspecialization}
Suppose that $\calL\in C(\X)(x)[\delta]$ and $n=\ord(\calL)$. Then there exists an ad-open subset $U$ of $\X(C)$ such that for any $c\in U$ and any irreducible right-hand divisor $P$ of $\calL^c$  with $\ord(P)<n$, $d(P)\leq n^3b(\calL)$.
\end{prop}
\begin{proof}
Let $A_\calL, T_i$ be as discussed before Proposition~\ref{prop:degreebound}. Let $U_1$ be a non-empty Zariski open subset of $\X(C)$ such that  for any $c\in U_1$, the following conditions hold:
\begin{enumerate}
\item [(1)]
$(A_\calL)^c\in \GL_n(C(x))$;
\item [(2)]$(A_\calL)^c$ is the companion matrix of $\calL^c$, i.e. $(A_\calL)^c=A_{\calL^c}$;
\item [(3)] For all $1\leq i \leq n$, $T_i^c$ is invertible and is the transformation matrix that transforms $\delta(Y)=(\wedge^i A_{\calL^c})Y$ into $(\calL^c)_{[i]}(y)=0$;
\item [(4)] For all $1\leq i \leq n$, $(\calL^c)_{[i]}=(\calL_{[i]})^c$
\end{enumerate}
For each $1\leq i \leq n$, due to Proposition 5.10 of \cite{Feng-Wibmer:differentialgaloisgropus}, there exists an exponential bound $N(\calL_{[i]})$ for $\calL_{[i]}$ and an ad-open subset $V_i$  of $\X(C)$ such that for all $c\in V_i$, $N(\calL_{[i]})$ is an exponential bound for $(\calL_{[i]})^c$.  Set $U=U_1\cap V_1\cap \dots \cap V_n$. For each $c\in U$, applying Proposition~\ref{prop:degreebound} to $\calL^c$ with $k=C$ yields that $d(P)\leq n^3b(\calL^c)$ for any irreducible right-hand divisor $P$ of $\calL^c$, where 
$$
 b(\calL^c)=\max_{1\leq i\leq n} \left\{2{n \choose i}\deg(T_i^c)+{n \choose i}\left({n\choose i}-1\right)N((\calL^c)_{[i]})\right\}.
$$
For each $1\leq i \leq n$, $\deg(T_i^c)\leq \deg(T_i)$ and $N(\calL_{[i]})$ is an exponential bound for $(\calL^c)_{[i]}$ for all $c\in U$ because $(\calL^c)_{[i]}=(\calL_{[i]})^c$ by the choice of $c$.
Therefore, $b(\calL^c)\leq b(\calL)$ for all $c\in U$. This concludes the proposition.
\end{proof}

\section{Proof of Theorem~\ref{thm:mainresult}}
 In this section, we shall prove our main result. Let's start with a lemma. Let $T=(T_1,\dots, T_m)$ be a vector with indeterminate entries. 
 \begin{lem}
 	\label{lm:solutionunderspecialization}
Let $f_1(T),\dots, f_s(T), g(T)$ be polynomials in $C[\X][T]$. Suppose that $g\neq 0$ and the system
$$f_1(T)=\dots=f_s(T)=0, g(T)\neq 0$$
has no solution in $\overline{C(\X)}^m$. Then there exists a non-empty Zariski subset $U$ of $\X(C)$ such that for every $c\in U$, $g^c\neq 0$ and the system
$$ f_1^c(T)=\dots=f^c_s(T)=0, g^c(T)\neq 0$$
has no solution in $C^m$.
\end{lem}
\begin{proof}
Let $y$ be a new indeterminate. Then the system 
$$f_1(T)=\dots=f_s(T)=0, g(T)\neq 0$$
has no solution in $\overline{C(\X)}^m$ if and only if the system 
\begin{equation}
\label{eqn:system2}
  f_1(T)=\dots=f_s(T)=yg(T)-1=0
 \end{equation}
has no solution in $\overline{C(\X)}^{m+1}$. Suppose that the system (\ref{eqn:system2}) has no solution in $\overline{C(\X)}^{m+1}$. By Hilbert's Nullstellensatz theorem, there exist $p_1,\dots, p_s, q \in C(\X)[T,y]$ such that 
$$
  \sum_{i=1}^s p_if_i+q(gy-1)=1.
$$ 
Let $h$ be a nonzero element in $C[\X]$ such that $hp_1,\dots, hp_s, hq\in C[\X][T,y]$, and let $U$ be the set of $\bfc\in \X(C)$ such that $h^c\neq 0$ and $g^c\neq 0$. Then for all $c\in U$, 
$$
   \sum_{i=1}^s (hp_i)^c f^c_i+(hq)^c (g^cy-1)=h^c\neq 0.
$$
This implies that for all $c\in U$, the system $f_1^c(T)=\dots=f_s^c(T)=g^cy-1=0$ has no solution in $C^{m+1}$. Equivalently, for all $c\in U$, the system $f_1^c(T)=\dots=f_s^c(T)=0, g^c(T)\neq 0$ has no solution in $C^m$.
\end{proof}

Suppose that $\calP,\calQ\in k(x)[\delta]$ and $\calQ\neq 0$. The Euclidean algorithm (see \cite{Ore1933}) implies that there exist $\calT,\calR\in k(x)[\delta]$ such that $\calP=\calT\calQ+\calR$ with $\calR=0$ or $\ord(\calR)<\ord(\calQ)$. The operator $\calR$ is called a remainder of $\calP$ with respect to $\calQ$, denoted by $\rem(\calP,\calQ)$. Furthermore, $\calQ$ is a right-hand divisor of $\calP$ if and only if $\rem(\calP,\calQ)=0$. The connection between the coefficients of the remainder $\calR$ and the coefficients of $\calP$ is described in the following lemma. Let $T$ be a vector with indeterminate entries. By setting $\delta(t)=0$ for any entry $t$ of $T$, $k(T)(x)$ becomes a differential field extension of $k(x)$.
\begin{lem}
\label{lm:remainder}
Suppose that $\calL\in k[T,x][\delta]$ and $\calP\in k(T)(x)[\delta]$. Assume further that 
$$
\calP=\delta^s+\frac{p_{s-1}}{q}\delta^{s-1}+\dots+\frac{p_0}{q}
$$
where $p_i,q\in k[T][x]$. Then there exist $\calQ,\calR\in k(T,x)[\delta]$ of the following form:
\begin{equation}
	\label{eqn:form}
  \frac{r_{\nu-1}}{q^m}\delta^{\nu-1}+\dots+\frac{r_0}{q^m}, m\geq 0, r_i\in k[T][x]
\end{equation}
such that $\calL=\calQ\calP+\calR$ and $\ord(\calR)<s$.
\end{lem}
\begin{proof}
	We will prove the lemma by induction on $n=\ord(\calL)$. If $n<s$ then set $\calQ=0,\calR=\calL$, and the lemma is obvious. Now, assume that $n\geq s$ and the assertion holds for operators with orders less than $n$.  A straightforward calculation yields that
	$$
	a\delta^{n-s}\calP=a\delta^n+\frac{b_{n-1}}{q^{n-s+1}}\delta^{n-1}+\dots+\frac{b_0}{q^{n-s+1}}
	$$
	where $a$ is the leading coefficient of $\calL$ and $b_i\in k[T][x]$.
    Hence $q^{n-s+1}(\calL-a\delta^{n-s}\calP)\in k[T,x][\delta]$ and its order is less than $n$. By the  induction hypothesis, there exist $\tilde{\calQ},\tilde{\calR}$ of the form (\ref{eqn:form}) such that 
    $$
      q^{n-s+1}(\calL-a\delta^{n-s}\calP)=\tilde{\calQ}\calP+\tilde{\calR}
    $$ 
    and $\ord(\tilde{\calR})<s$. Then $$
    \calL=\left(a\delta^{n-s}+\frac{\tilde{\calQ}}{q^{n-s+1}}\right)\calP+\frac{\tilde{\calR}}{q^{n-s+1}}.
    $$
    So $a\delta^{n-s}+\frac{\tilde{\calQ}}{q^{n-s+1}}$ and $\frac{\tilde{\calR}}{q^{n-s+1}}$ are the operators as required.
\end{proof}

We are now ready to prove our main result.
\begin{proof}[Proof of Theorem~\ref{thm:mainresult}]
 Let $U_1$ be an ad-open subset of $\X(C)$ such that for all $c\in U_1$ and all irreducible right-hand divisors $P$ of $\calL^c$ with $\ord(P)<\ord(\calL^c)$, $d(P)\leq n^3b(\calL)$. Such $U_1$ exists due to Proposition~\ref{prop:boundunderspecialization}.  We shall show that for each $1\leq s \leq n-1$, there exists a non-empty Zariski open subset $V_s$ of $\X(C)$ such that for all $c\in U_1\cap V_s$, $\calL^c$ has no irreducible right-hand divisor of order $s$. Set $\nu=n^4b(\calL)$. Let 
$$
 T=(t_{0,0},t_{1,0},\dots,t_{s,\nu})
$$ 
be a vector with indeterminate entries $t_{i,j}$ and set $k=C(\X)(T)$. By setting $\delta(t_{i,j})=0$ for all $i,j$, $k(x)$ becomes a differential extension field of $C(\X)(x)$. Set $q=\sum_{j=0}^\nu t_{s,j}x^j$ and
$$
   \calP=\delta^s+\left(\frac{\sum_{j=0}^\nu t_{s-1,j}x^j}{q}\right)\delta^{s-1}+\dots+\left(\frac{\sum_{j=0}^\nu t_{0,j}x^j}{q}\right).
$$
Let $a$ be an element in $C[\X][x]$ such that $a\calL\in C[\X][x][\delta]$. 
Due to Lemma~\ref{lm:remainder}, there exist $\tilde{\calQ},\tilde{\calR}$ of the form (\ref{eqn:form}) such that 
$a\calL=\tilde{\calQ}\calP+\tilde{\calR}$ and $\ord(\calR)<s$. Therefore
$$
   \calL=\left(\frac{p_{n-s}}{aq^{m_1}}\delta^{n-s}+\dots+\frac{p_0}{aq^{m_1}}\right)\calP+\frac{r_{s-1}}{aq^{m_2}}\delta^{s-1}+\dots+\frac{r_0}{aq^{m_2}},
$$
where $p_i,r_j\in C(\X)[T][x]$. Multiplying $a$ by a suitable element in $C[\X]$, one may assume that $p_i,r_j\in C[\X][T][x]$. Each element of $\X(C)$ can be lifted to a $C$-homomorphism from $C[\X][T]$ to $C[T]$ by assigning $c(t_{i,j})=t_{i,j}$ for all $i,j$. For the sake of notation, we still use $c$ to denote the lifted homomorphism. Note that  $c(\delta(f))=\delta(c(f))$ for any $f\in C[\X][x]$, and $q^c=q,\calP^c=\calP$. For $c\in \X(C)$ such that $a^c\neq 0$, one has that
\begin{equation}
	\label{eqn:remaindunderspecialization}
   \calL^c=\left(\frac{p^c_{n-s}}{a^cq^{m_1}}\delta^{n-s}+\dots+\frac{p^c_0}{a^cq^{m_1}}\right)\calP+\frac{r^c_{s-1}}{a^cq^{m_2}}\delta^{s-1}+\dots+\frac{r^c_0}{a^cq^{m_2}}.
\end{equation}
On the other hand, let $\bfc\in \overline{C(\X)}^{(s+1)(\nu+1)}$ satisfy $q(\bfc)\neq 0$, where $q(\bfc)$ denotes the element in $\overline{C(\X)}[x]$ obtained by substituting $\bfc$ into $T$. One has that 
$$
    \calL=\left(\frac{p_{n-s}(\bfc)}{aq(\bfc)^{m_1}}\delta^{n-s}+\dots+\frac{p_0(\bfc)}{aq(\bfc)^{m_1}}\right)\calP(\bfc)+\frac{r_{s-1}(\bfc)}{aq(\bfc)^{m_2}}\delta^{s-1}+\dots+\frac{r_0(\bfc)}{aq(\bfc)^{m_2}}.
$$
Similarly, after replacing $T$ with $\bfc\in C^{(s+1)(\nu+1)}$ satisfying $p(\bfc)\neq 0$, the equality (\ref{eqn:remaindunderspecialization}) still holds.  

Let $W$ be the set of the coefficients of $r_{s-1},\dots,r_0$, viewed as polynomials in $x$. Then $W\subset C(\X)[T]$. Since $\calL$ is irreducible over $\overline{C(\X)}(x)$, for any $0\leq j \leq \nu$, the system 
$$
 S_j=\{f(T)=0, \forall f\in W, t_{s,j}\neq 0\}
$$ 
has no solution in $\overline{C(\X)}^{(s+1)(\nu+1)}$. Let $V_s$ be a non-empty Zariski open subset of $\X(C)$ such that for any $c\in V_s$, all systems $S_0^c,S_1^c,\dots,S_\nu^c$ have no solution in $C^{(s+1)(\nu+1)}$, where $S^c_j=\{f^c(T), \forall f\in W, t_{s,j}\neq 0\}$. Such $V_s$ exists due to Lemma~\ref{lm:solutionunderspecialization}. Set $U=U_1\cap V_1\cap \dots \cap V_{n-1}$. We claim that $U$ is what we seek. Suppose that $c\in U$ and $\calL^c$ has an irreducible right-hand divisor $P$ of order $s$ with $1\leq s \leq n-1$. Without loss of generality, we may assume that $P$ is monic. By Proposition~\ref{prop:degreebound}, $d(P)\leq n^3b(\calL)$. By taking a common denominator of the coeficients of $P$, each coefficient of $P$ can be rewritten as the quotient of two polynomials in $x$ with degrees not greater than $\nu$. In other words, we may write
$$
    P=\delta^s+\left(\frac{\sum_{j=0}^\nu c_{s-1,j}x^j}{\sum_{j=0}^\nu c_{s,j}x^j}\right)\delta^{s-1}+\dots+\left(\frac{\sum_{j=0}^\nu c_{0,j}x^j}{\sum_{j=0}^\nu c_{s,j}x^j}\right)
$$
where $c_{i,j}\in C$ and not all $c_{s,0},\dots,c_{s,\nu}$ are zero. Set $\bfc=(c_{0,0},c_{1,0},\dots,c_{s,\nu})$. Since $q(\bfc)=\sum_{j=0}^\nu c_{s,j}x^j\neq 0$ and $P$ is a right-hand divisor of $\calL^c$, substituting $\bfc$ into $T$ in (\ref{eqn:form}) yields that 
$$
   r^c_{s-1}(\bfc)=\dots=r_0^c(\bfc)=0.
$$ 
This implies that $f^c(\bfc)=0$ for all $f\in W$. As there exists $0\leq j_1\leq \nu$ such that $c_{s,j_1}\neq 0$, $\bfc$ is a solution of the system $S_{j_1}^c$, contradicting the choice of $c$ such that $S^c_j$ has no solution in $C^{(s+1)(\nu+1)}$ for all $0\leq j\leq \nu$. This proves our claim which concludes the theorem.
\end{proof}

\bibliographystyle{plain}

\end{document}